\documentclass[12pt,a4paper]{amsart}
\usepackage{amsmath}
\usepackage{drpack}
\usepackage[english]{babel}
\usepackage{verbatim}
\usepackage[shortlabels]{enumitem}
\usepackage{tikz-cd}

\newtheoremstyle{mio}%
{}{} 
{\itshape}{} 
{\bfseries}{.}{ } 
{#1 #2\thmnote{~\mdseries(#3)}} 
\theoremstyle{mio}
\newtheorem{teor}{Theorem}[section]
\newtheorem{cor}[teor]{Corollary}
\newtheorem{prop}[teor]{Proposition}
\newtheorem{lemma}[teor]{Lemma}

\newtheoremstyle{definition2}%
{}{} 
{}{} 
{\bfseries}{.}{ } 
{#1 #2\thmnote{\mdseries~ #3}} 
\theoremstyle{definition2}
\newtheorem{ex}[teor]{Example}
\newtheorem{oss}[teor]{Remark}

\title{Isolated points of the Zariski space}
\author{Dario Spirito}
\date{\today}
\address{Dipartimento di Matematica ``Tullio Levi-Civita'', Universit\`a degli Studi di Padova, Padova, Italy}
\email{spirito@math.unipd.it}

\newcommand{\B}{\mathcal{B}}
\newcommand{\V}{\mathcal{V}}
\newcommand{\D}{\mathcal{D}}
\newcommand{\inverse}{\mathrm{inv}}
\newcommand{\cons}{\mathrm{cons}}

\newcommand{\Cl}{\mathrm{Cl}}
\newcommand{\est}{\mathcal{E}}

\newcommand{\mm}{\mathfrak{m}}
\newcommand{\gen}{\mathrm{gen}}

\begin{document}

\begin{abstract}
Let $D$ be an integral domain and $L$ be a field containing $D$. We study the isolated points of the Zariski space $\Zar(L|D)$, with respect to the constructible topology. In particular, we completely characterize when $L$ (as a point) is isolated and, under the hypothesis that $L$ is the quotient field of $D$, when a valuation domain of dimension $1$ is isolated; as a consequence, we find all isolated points of $\Zar(D)$ when $D$ is a Noetherian domain and, under the hypothesis that $D$ and $D'$ are Noetherian, local and countable, we characterize when $\Zar(D)$ and $\Zar(D')$ are homeomorphic. We also show that if $V$ is a valuation domain and $L$ is transcendental over $V$ then the set of extensions of $V$ to $L$ has no isolated points.
\end{abstract}

\maketitle

\section{Introduction}
Let $D$ be an integral domain with quotient field $K$, and let $L$ be a field containing $K$. The \emph{Zariski space} of $L$ over $D$, denoted by $\Zar(L|D)$, is the set of all valuation rings containing $D$ and having quotient field $L$. O. Zariski introduced this set (under the name \emph{abstract Riemann surface}) and endowed it with a natural topology (later called the \emph{Zariski topology}) during its study of resolution of singularities; in particular, he used the compactness of the Zariski space to reduce the problem of gluing infinitely many projective models to the gluing of only finitely many of them \cite{zariski_sing,zariski_comp}. Later on, it was showed that $\Zar(L|D)$ enjoys even deeper topological properties: in particular, it is a \emph{spectral space}, meaning that there is always a ring $R$ such that $\Spec(R)$ (endowed with the Zariski topology) is homeomorphic to $\Zar(L|D)$, and an example of such an $R$ can be find using the Kronecker function ring construction \cite{dobbs_fedder_fontana,fontana_krr-abRs,fifolo_transactions}. Beyond being a very natural example of a spectral space ``occurring in nature'', the Zariski topology can also be used, for example, to study representation of integral domains as intersection of overrings \cite{olberding_noetherianspaces,olberding_affineschemes,olberding_topasp}, or in real and rigid algebraic geometry \cite{hub-kneb,schwartz-compactification}.

As a spectral space, two other topologies can be constructed on $\Zar(L|D)$ starting from the Zariski topology: the \emph{inverse} and the \emph{constructible} (or \emph{patch}) topology. Both of them give rise to spectral spaces (in particular, they are compact); furthermore, the constructible topology gains the property of being Hausdorff, and plays an important role in the topological characterization of spectral spaces (see for example Hochster's article \cite{hochster_spectral}). The constructible topology can also be studied through ultrafilters \cite{finocchiaro-ultrafiltri}, and this point of view allows to give many examples of spectral spaces, for example by finding them inside other spectral spaces (see \cite[Example 2.2(1)]{olberding_topasp} for some very general constructions, \cite{localizzazioni} for examples in the overring case, and \cite{topological-cons,spettrali-eab} for examples in the setting of semistar operations).

In this paper, we want to study the points of $\Zar(L|D)$ that are isolated, with respect to the constructible topology. Our starting point is a new interpretation of a result about the compactness of spaces in the form $\Zar(K|D)\setminus\{V\}$ \cite[Theorem 3.6]{ZarNoeth}, where $K$ is the quotient field of $D$: in particular, we show that if $V$ is isolated then $V$ is the integral closure of $D[x_1,\ldots,x_n]_M$, where $x_1,\ldots,x_n\in L$ and $M$ is a maximal ideal of $D[x_1,\ldots,x_n]$ (Theorem \ref{teor:x1xn}). Through this result, we characterize when $L$ is an isolated point of $\Zar(L|D)^\cons$ (i.e., $\Zar(L|D)$ endowed with the constructible topology; Proposition \ref{prop:isolated-dim0}) and, under the hypothesis that $L=K$ is the quotient field of $D$, when the one-dimensional valuation overrings are isolated (Theorem \ref{teor:ic-dim1}). 

In Section \ref{sect:noeth}, we study the isolated points of the constructible topology when $D$ is a Noetherian domain and $L=K$ is its quotient field. Theorem \ref{teor:noeth} gives a complete characterization: $V\in\Zar(K|D)=\Zar(D)$ is isolated if and only if the center $P$ of $V$ on $D$ has height at most $1$ and $P$ is contained in only finitely many minimal primes; in particular, this cannot happen if $D$ is local and of dimension at least $3$. In the countable case, we also give a complete characterization of when $\Zar(D)^\cons\simeq\Zar(D')^\cons$ under the hypothesis that $D$ and $D'$ are Noetherian and local (Theorem \ref{teor:Zarsimeq}).

The last two sections of the paper explore the case of extension of valuations. Section \ref{sect:Dfield} studies the case where $D$ itself is a field: in particular, we show that if the transcendence degree of $L$ over $D$ is at least $2$ then $\Zar(L|D)^\cons$ has no isolated points, improving \cite[Theorem 4.45]{preordered-groups}. In Section \ref{sect:extval}, we show that if $V$ is a valuation domain that is not a field and $K(X)$ is the field of rational functions, then the set of extensions of $V$ to $K(X)$ has no isolated points (Theorem \ref{teor:ext-K(X)}), and as a consequence we further extend \cite[Theorem 4.45]{preordered-groups} to $\Zar(L|D)^\cons$ when $D$ is an arbitrary integral domain (Theorem \ref{teor:trdeg2-field} and Corollary \ref{cor:LD-trdeg2}).

\section{Notation and preliminaries}
Throughout the paper, all rings will be commutative, unitary and will have no zero-divisors (that is, they are integral domains). We usually denote by $D$ such a domain and by $K$ its quotient field; we use $\overline{D}$ to denote the integral closure of $D$ in $K$.

\subsection{Spectral spaces}
A topological space $X$ is \emph{spectral} if it is homeomorphic to the prime spectrum of a ring, endowed with the Zariski topology; spectral spaces can also be characterized in a purely topological way (see \cite{hochster_spectral} and \cite{spectralspaces-libro}). Among their properties, spectral spaces are always compact and have a basis of open and compact sets. If $\Delta\subseteq X$, we denote by $\Cl(\Delta)$ the closure of $\Delta$. The topology of $X$ induces an order such that $x\leq y$ if and only if $y\in\Cl(x)$. If $Y\subseteq X$, the \emph{closure under generization} of $Y$ if the set $Y^\gen:=\{x\in X\mid x\leq y$ for some $y\in Y\}$, where $\leq$ is the order induced by the topology, and $Y$ is \emph{closed by generizations} if $Y=Y^\gen$.

If $X$ is a spectral space, the \emph{inverse topology} on $X$ is the coarsest topology such that the open and compact subsets of $X$ are closed. We denote by $X^\inverse$ the space $X$, endowed with the inverse topology. A subset $Y\subseteq X$ is closed in the inverse topology if and only if it is compact in the starting topology and closed by generizations; in particular, if $Y$ is compact in the starting topology then its closure in the inverse topology is $Y^\gen$.

If $X$ is a spectral space, the \emph{constructible topology} (or \emph{patch topology}) on $X$ is the coarsest topology such that the open and compact subsets of $X$ are both open and closed. We denote by $X^\cons$ the space $X$, endowed with the constructible topology; if $Y\subseteq X$, we denote by $Y^\cons$ the subset $Y$ considered with respect to the constructible topology, and by $\Cl^\cons(Y)$ the closure of $Y$ in $X^\cons$. If $Y=\Cl^\cons(Y)$, then $Y$ is compact in the starting topology; conversely, if $Y$ is closed in the starting topology or in the inverse topology, then it is closed also in the constructible topology.

Both $X^\inverse$ and $X^\cons$ are spectral spaces, and in particular compact spaces; moreover, $X^\cons$ is Hausdorff and zero-dimensional.

A map $f:X\longrightarrow Y$ of spectral spaces is a \emph{spectral map} if $f^{-1}(\Omega)$ is open and compact for every open and compact subset $\Omega$ of $Y$; in particular, a spectral map is continuous. If $f$ is both spectral and closed, then it is also proper, and in particular $f^{-1}(\Omega)$ is compact for every compact subset $\Omega$ of $Y$ \cite[5.3.7(i)]{spectralspaces-libro}. If $f:X\longrightarrow Y$ is a spectral map, then it is spectral also when $X$ and $Y$ are both endowed with the inverse topology, and when they are both endowed with the constructible topology \cite[Theorem 1.3.21]{spectralspaces-libro}. In the latter case, $f$ is also closed, since it is a continuous map between Hausdorff compact spaces.

\subsection{Isolated points}
If $X$ is a topological space, a point $p\in X$ is \emph{isolated} in $X$ if $\{p\}$ is an open set. If $X$ has no isolated points, then $X$ is said to be \emph{perfect}. The set of points that are not isolated in $X$ is a closed set, called the \emph{derived set} of $X$.

If $\Omega\subseteq X$ and $p\in\Omega$ is isolated in $X$, then $p$ is also isolated in $\Omega$; if $\Omega$ is open, then $p$ is isolated in $X$ if and only if $p$ is isolated in $\Omega$.

\subsection{Valuation domains}

A \emph{valuation domain} is an integral domain $V$ such that, for every $x\neq 0$ in the quotient field of $V$, at least one of $x$ and $x^{-1}$ is in $V$. Any valuation domain is local; we denote the maximal ideal of $V$ by $\mathfrak{m}_V$. If $L$ is a field containing the quotient field $K$ of $V$, an \emph{extension} of $V$ to $L$ is a valuation domain $W$ having quotient field $L$ such that $W\cap K=V$. We denote the set of extension of $V$ to $L$ by $\est(L|V)$; this set is always nonempty (see e.g. \cite[Theorem 20.1]{gilmer}).

If $D$ is an integral domain and $L$ is a field containing $D$, the \emph{Zariski space} (or \emph{Zariski-Riemann space}) of $L$ over $D$, denoted by $\Zar(L|D)$, is the set of all valuation domains containing $D$ and having quotient field $L$. The Zariski space $\Zar(L|D)$ is always nonempty. When $L$ is the quotient field of $D$, we denote $\Zar(L|D)$ simply by $\Zar(D)$, and we call its elements the \emph{valuation overrings} of $D$.\footnote{An \emph{overring} of $D$ is, more generally, a ring contained between $D$ and its quotient field.} If $D'$ is the integral closure of $D$ in $L$, then $\Zar(L|D)=\Zar(L|D')$; in particular, $\Zar(D)=\Zar(\overline{D})$. A valuation ring in $\Zar(L|D)$ is \emph{minimal} if it is minimal with respect to containment.

The Zariski-Riemann space $\Zar(L|D)$ can be endowed with a natural topology, called the \emph{Zariski topology}, which is the topology generated by the basic open sets
\begin{equation*}
\B(x_1,\ldots,x_n):=\{V\in\Zar(L|D)\mid x_1,\ldots,x_n\in V\},
\end{equation*}
as $x_1,\ldots,x_n$ range among the elements of $L$; we use the notation $\B^L(x_1,\ldots,x_n)$ if we need to underline the field $L$. Under this topology, $\Zar(L|D)$ is a spectral space whose order is the opposite of the containment order \cite{fontana_krr-abRs,dobbs_fedder_fontana}; in particular, the minimal valuation rings in $\Zar(L|D)$ are maximal with respect to the order induced by the Zariski topology. As a spectral space, we can define the inverse and the constructible topology on $\Zar(L|D)$; a set $\Delta\subseteq\Zar(L|D)$ is closed with respect to the inverse topology if and only if it is compact with respect to the Zariski topology and $\Delta=\{W\in\Zar(L|D)| W\supseteq V$ for some $V\in\Delta\}$ \cite[Remark 2.2 and Proposition 2.6]{fifolo_transactions}.

Since $\B(z_1,\ldots,z_n)=\B(z_1)\cap\cdots\cap\B(z_n)$ for every $z_1,\ldots,z_n\in L$, a basis of the constructible topology of $\Zar(L|D)$ is the family of the sets in the form $\B(x_1,\ldots,x_n)\cap\B(y_1)^c\cap\cdots\cap\B(y_m)^c$, as $x_1,\ldots,x_n,y_1,\ldots,y_m$ range in $L$. In particular, $V$ is isolated in $\Zar(L|D)^\cons$ if and only if
\begin{equation*}
\begin{aligned}
\{V\} & =\B(x_1,\ldots,x_n)\cap\B(y_1)^c\cap\cdots\cap\B(y_m)^c=\\
 & =\Zar(L|D[x_1,\ldots,x_n])\cap\B(y_1)^c\cap\cdots\cap\B(y_m)^c
\end{aligned}
\end{equation*}
for some $x_1,\ldots,x_n,y_1,\ldots,y_m\in L$.

If $L'\subseteq L$ is a field extension and $D\subseteq L'$, we have a restriction map
\begin{equation*}
\begin{aligned}
\rho\colon\Zar(L|D) & \longrightarrow\Zar(L'|D),\\
V & \longmapsto V\cap L',
\end{aligned}
\end{equation*}
which is spectral in both the Zariski and the constructible topology; therefore, it is spectral and closed with respect to the constructible topology (on both sets). In particular, if $V\in\Zar(L'|D)$, then $\est(L|V)=\rho^{-1}(V)$; therefore, $\est(L|V)$ is always closed in $\Zar(L|D)^\cons$, and in particular it is compact both in the Zariski and the constructible topology.

\medskip

Since, by definition, the spectrum $\Spec(D)$ is a spectral space (when endowed with the Zariski topology), we can define the inverse and the constructible topology also on $\Spec(D)$. For every ideal $I$ of $D$, set $\V(I):=\{P\in\Spec(D)\mid I\subseteq P\}$ and $\D(I):=\Spec(D)\setminus\V(I)$: then, a basis of $\Spec(D)^\cons$ is given by the intersections $\V(aD)\cap\D(I)$, as $a$ ranges in $D$ and $I$ among the finitely generated ideals of $D$ \cite[Theorem 12.1.10(iv)]{spectralspaces-libro}.

For every field $L$, we can define a map
\begin{equation*}
\begin{aligned}
\gamma\colon\Zar(L|D) & \longrightarrow\Spec(D),\\
V & \longmapsto\mathfrak{m}_V\cap D,
\end{aligned}
\end{equation*}
which is called the \emph{center map}. When $\Zar(L|D)$ and $\Spec(D)$ are endowed with the Zariski topology, $\gamma$ is continuous (\cite[Chapter VI, \textsection 17, Lemma 1]{zariski_samuel_II} or \cite[Lemma 2.1]{dobbs_fedder_fontana}), spectral, surjective (this follows, for example, from \cite[Theorem 5.21]{atiyah} or \cite[Theorem 19.6]{gilmer}) and closed \cite[Theorem 2.5]{dobbs_fedder_fontana}, so in particular it is proper. Therefore, $\gamma$ is a spectral map also when $\Zar(L|D)$ and $\Spec(D)$ are endowed with their respective constructible topologies.

\section{General results}
We begin by establishing some general criteria to determine which valuation domains are isolated in $\Zar(D)$.

Let $D$ be an integral domain: we say that a prime ideal $P$ of $D$ is \emph{almost essential} if there is a unique valuation overring of $D$ having center $P$; equivalently, $P$ is almost essential if and only if the integral closure of $D_P$ is a valuation domain $V$. When this happens, we say that $V$ is an \emph{almost essential valuation overring} of $D$. When $D_P$ itself is a valuation overring, we say that $P$ is an \emph{essential prime} of $D$ and that $D_P$ is an \emph{essential valuation overring} of $D$.

In the context of almost isolated primes and valuation overrings, isolated valuation rings correspond to isolated prime ideals.
\begin{prop}\label{prop:essential}
Let $D$ be an integral domain, and let $P$ be an almost essential prime ideal of $D$; let $V$ be the valuation overring with center $P$. Then, $V$ is isolated in $\Zar(D)^\cons$ if and only if $P$ is isolated in $\Spec(D)^\cons$.
\end{prop}
\begin{proof}
Let $\gamma:\Zar(D)\longrightarrow\Spec(D)$ be the center map. If $P$ is isolated in $\Spec(D)^\cons$, then $\{P\}$ is open and thus, as $\gamma$ is continuous, $\{V\}=\gamma^{-1}(\{P\})$ is open in $\Zar(D)^\cons$, i.e., $V$ is isolated. Conversely, if $V$ is isolated then $\Zar(D)\setminus\{V\}$ is closed, with respect to the constructible topology, and thus $\gamma(\Zar(D)\setminus\{V\})=\Spec(D)\setminus\{P\}$ is closed in $\Spec(D)^\cons$. Hence, $\{P\}$ is open and $P$ is isolated in $\Spec(D)^\cons$, as claimed.
\end{proof}

\begin{cor}
Let $D$ be a Pr\"ufer domain, and let $V$ be a valuation overring of $D$ with center $P$. Then, $V$ is isolated in $\Zar(D)^\cons$ if and only if $P$ is isolated in $\Spec(D)^\cons$. In particular, $\Zar(D)^\cons$ is perfect if and only if $\Spec(D)^\cons$ is perfect.
\end{cor}
\begin{proof}
Since $D$ is a Pr\"ufer domain, every valuation overring is essential. The claim follows from Proposition \ref{prop:essential}.
\end{proof}

In general, almost essential valuation overrings are rare; for example, if $D$ is Noetherian, no prime ideal of height $2$ or more can be almost essential. For this reason, we need more general results; the first step is connecting isolated valuation rings with compactness.
\begin{prop}\label{prop:minimale}
Let $X$ be a spectral space, and let $x$ be a maximal element with respect to the order induced by the topology. Then, the following are equivalent:
\begin{enumerate}[(i)]
\item\label{prop:minimale:isolato} $x$ is isolated in $X^\cons$;
\item\label{prop:minimale:comp} $X\setminus\{x\}$ is compact, with respect to the starting topology;
\item\label{prop:minimale:clinv} $X\setminus\{x\}$ is closed, with respect to the inverse topology.
\end{enumerate}
\end{prop}
\begin{proof}
Let $Y:=X\setminus\{x\}$.

The equivalence of \ref{prop:minimale:comp} and \ref{prop:minimale:clinv} follows from the fact that $Y$ is closed by generizations.

If \ref{prop:minimale:isolato} holds, then $\{x\}$ is an open set in the constructible topology, and thus $Y$ is closed; since $X^\cons$ is compact, it follows that $Y$ is compact in the constructible topology and thus also in the Zariski topology (which is coarser). Thus, \ref{prop:minimale:comp} holds.

Conversely, if \ref{prop:minimale:clinv} holds, then $Y$ is closed also in the constructible topology; hence, $\{x\}$ is open and $x$ is isolated. Thus, \ref{prop:minimale:isolato} holds.
\end{proof}

In particular, the previous proposition applies when $X=\Zar(L|D)$ and $V$ is a minimal valuation overring. In this case, the fact that $\Zar(L|D)\setminus\{V\}$ is compact has very strong consequences.
\begin{teor}\label{teor:x1xn}
Let $D$ be an integrally closed domain and let $V\in\Zar(L|D)$. Then, the following are equivalent.
\begin{enumerate}[(i)]
\item\label{teor:x1xn:isol} $V$ is isolated in $\Zar(L|D)^\cons$;
\item\label{teor:x1xn:max} there are $x_1,\ldots,x_n\in L$ and a maximal ideal $M$ of $D[x_1,\ldots,x_n]$ such that $V$ is the integral closure of $D[x_1,\ldots,x_n]_M$ and $M$ is isolated in $\Spec(D[x_1,\ldots,x_n])^\cons$;
\item\label{teor:x1xn:prime} there are $x_1,\ldots,x_n\in L$ and a prime ideal $P$ of $D[x_1,\ldots,x_n]$ such that $V$ is the integral closure of $D[x_1,\ldots,x_n]_P$ and $P$ is isolated in $\Spec(D[x_1,\ldots,x_n])^\cons$.
\end{enumerate}
\end{teor}
\begin{proof}
Let $X$ be an indeterminate over $D$, and let $R:=D+XL[[X]]$. By the reasoning in the proof of \cite[Proposition 3.3]{ZarNoeth2} (or by Lemma \ref{lemma:quoziente} below) the Zariski space $\Zar(L|D)^\cons$ is homeomorphic to $(\Zar(R)\setminus\{L[[X]]\})^\cons$, which is open in $\Zar(R)^\cons$; in particular, a $W\in\Zar(L|D)$ is isolated with respect to the constructible topology if and only if $W+XL[[X]]$ is isolated in $\Zar(R)^\cons$. Therefore, without loss of generality we can suppose that $L$ is the quotient field of $D$. 

\ref{teor:x1xn:isol} $\Longrightarrow$ \ref{teor:x1xn:prime} Since $V$ is isolated, there are $x_1,\ldots,x_k,y_1,\ldots,y_m\in L$ such that $\{V\}=\Zar(D[x_1,\ldots,x_k])\cap\B(y_1)^c\cap\cdots\cap\B(y_m)^c$. In particular, $V$ is a minimal valuation overring of $D[x_1,\ldots,x_k]$. By Proposition \ref{prop:minimale}, $\Zar(D[x_1,\ldots,x_k])\setminus\{V\}$ is compact, with respect to the Zariski topology; therefore, by \cite[Theorem 3.6]{ZarNoeth}, there are $x_{k+1},\ldots,x_n\in L$ such that $V$ is the integral closure of $D[x_1,\ldots,x_k][x_{k+1},\ldots,x_n]_M=D[x_1,\ldots,x_n]_M$, for some maximal ideal $M$ of $D[x_1,\ldots,x_n]$. Hence, $M$ is almost essential in $D[x_1,\ldots,x_n]$, and thus, by Proposition \ref{prop:essential}, $M$ must be isolated in $\Spec(D[x_1,\ldots,x_n])^\cons$. Hence, \ref{teor:x1xn:max} holds.

\ref{teor:x1xn:max} $\Longrightarrow$ \ref{teor:x1xn:prime} is obvious.

\ref{teor:x1xn:prime} $\Longrightarrow$ \ref{teor:x1xn:isol} The set $\Zar(D[x_1,\ldots,x_n])=\B(x_1,\ldots,x_n)$ is open in the constructible topology, and thus $V$ is isolated in $\Zar(D)^\cons$ if and only if it is isolated in $\Zar(D[x_1,\ldots,x_n])^\cons$. By hypothesis, $P$ is almost essential for $D[x_1,\ldots,x_n]$, and thus by Proposition \ref{prop:essential} the integral closure $V$ of $D[x_1,\ldots,x_n]_P$ is isolated, as claimed.
\end{proof}

\section{Dimension 0}\label{sect:dim0}
In this section, we study when the field $L$ is isolated in $\Zar(L|D)^\cons$. If $L$ is the quotient field of $D$, then $L$ is an essential valuation overring of $D$, and thus one can reason through Proposition \ref{prop:essential}; however, it is possible to use a more general approach.

A domain $D$ with quotient field $K$ is said to be a \emph{Goldman domain} (or a \emph{G-domain}) if $K$ is a finitely generated $D$-algebra, or equivalently if $K=D[u]$ for some $u\in K$.
\begin{prop}\label{prop:isolated-dim0}
Let $D$ be an integral domain with quotient field $K$, and let $L$ be a field extension of $K$. Then, $L$ is isolated in $\Zar(L|D)^\cons$ if and only if $D$ is a Goldman domain and $K\subseteq L$ is an algebraic extension.
\end{prop}
\begin{proof}
Suppose first that the two conditions hold. Then, $K=D[u]$ for some $u\in K$; since $K\subseteq L$ is algebraic, it follows that $\B(u)=\Zar(L|K)=\{L\}$. Hence, $L$ is isolated in $\Zar(L|D)^\cons$.

Conversely, suppose that $L$ is isolated. By Theorem \ref{teor:x1xn}, there are $x_1,\ldots,x_n\in L$ such that $L$ is the integral closure of $D[x_1,\ldots,x_n]_M$ for some maximal ideal $M$; since $M$ must have height $0$, $F:=D[x_1,\ldots,x_n]$ must be a field such that $F\subseteq L$ is algebraic.

Suppose that $F$ is transcendental over $K$: then, we can take a transcendence basis $y_1,\ldots,y_k$ of $F$ over $K$. By construction, $F$ is algebraic over $D[y_1,\ldots,y_k]$; since $F$ is a field, it is a Goldman domain, and thus by \cite[Theorem 22]{kaplansky} so should be $D[y_1,\ldots,y_k]$, against \cite[Theorem 21]{kaplansky}. Thus $F$ is algebraic over $K$. Applying again \cite[Theorem 22]{kaplansky} to the algebraic extension $D\subset F$, we see that $D$ is a Goldman domain; furthermore, $L$ is algebraic over $F$ and thus over $K$. The claim is proved.
\end{proof}

The previous result can be used to give some necessary conditions for $V$ to be isolated. We premise a lemma.
\begin{lemma}\label{lemma:quoziente}
Let $D$ be an integral domain, $L$ be a field containing $D$, and let $W\in\Zar(L|D)$. Let $\pi:W\longrightarrow W/\mm_W$ be the quotient map. Then, the map
\begin{equation*}
\begin{aligned}
\overline{\pi}\colon\{Z\in\Zar(L|D)\mid Z\subseteq W\} & \longrightarrow\Zar(W/\mm_W|D/(\mm_W\cap D)),\\
Z & \longmapsto \pi(Z)\\
\end{aligned}
\end{equation*}
is a homeomorphism, when both sets are endowed with either the Zariski or the constructible topology.
\end{lemma}
\begin{proof}
Let $Z\in\Zar(L|D)$: then, $\ker\pi=\mm_W\subseteq Z$ since $Z$ and $W$ are valuation domains with the same quotient field and $Z\subseteq W$. Hence, $\pi(Z)=Z/\mm_W$ is a valuation ring containing $D/(\mm_W\cap D)$; moreover, since $W$ is a localization of $V$, $W/\mm_W$ is a localization of $Z/\mm_W$ and thus $W/\mm_W$ is the quotient field of $\pi(Z)$. Hence, $\overline{\pi}$ is well-defined.

Moreover, if $Z'\in\Zar(W/\mm_W|D/M)$, then $Z:=\pi^{-1}(Z')$ is the pullback of $Z'$ along the quotient $W\longrightarrow W/\mm_W$. Thus, $Z$ is a valuation domain by \cite[Proposition 1.1.8(1)]{fontana_libro}, and its quotient field is $L$ by \cite[Lemma 1.1.4(10)]{fontana_libro}. Hence $\overline{\pi}$ is surjective. Furthermore, if $Z\in\Zar(L|D)$, then $\ker\pi\subseteq Z$ and thus $\pi^{-1}(\pi(Z))=Z$; hence, $\overline{\pi}$ is bijective.

Let now $x\in W/\mm_W$. Then, $Z\in\overline{\pi}^{-1}(\B(x))$ if and only if $x\in\pi(Z)$. Since $\ker\pi\subseteq Z$, this happens if and only if $Z$ contains all of $\pi^{-1}(x)$; thus, for every $y\in\pi^{-1}(x)$, we have $\overline{\pi}^{-1}(\B(x))=\B(y)$, and likewise $\overline{\pi}(\B(x))=\B(\pi(x))$ for every $x\in L$. Hence, $\overline{\pi}$ is continuous and open when both $\{Z\in\Zar(L|D)\mid Z\subseteq W\}$ and $\Zar(W/\mm_W|D/(\mm_W\cap D))$ are endowed with the Zariski topology, and thus it is a homeomorphism. It follows that it is also a homeomorphism when both sets are endowed with the constructible topology, as claimed.
\end{proof}

\begin{prop}\label{prop:estresfield-trasc}
Let $V\in\Zar(D)$ be a valuation domain with center $P$ on $D$. If $V$ is isolated in $\Zar(D)^\cons$, then the field extension $D_P/PD_P\subseteq V/\mathfrak{m}_V$ is algebraic.
\end{prop}
\begin{proof}
Consider $\Delta:=\{W\in\Zar(D)\mid W\subseteq V\}$. Since $\mm_V\cap D=P$, by Lemma \ref{lemma:quoziente}, the quotient map $V\longrightarrow V/\mathfrak{m}$ induces a homeomorphism between $\Delta^\cons$ and $\Zar(V/\mathfrak{m}_V|D/P)^\cons$, and thus $V/\mathfrak{m}$ is isolated in $\Zar(V/\mathfrak{m}_V|D/P)^\cons$. Let $F$ be the quotient field of $D/P$: then, $F=(D/P)_{P/P}=D_P/PD_P$. By Proposition \ref{prop:isolated-dim0}, $F\subseteq V/\mathfrak{m}_V$ must be algebraic; the claim follows.
\end{proof}

\begin{cor}
Let $D$ be an integral domain, let $\gamma:\Zar(D)\longrightarrow\Spec(D)$ be the center map and let $V\in\Zar(D)$. If $V$ is isolated in $\Zar(D)^\cons$, then $V$ is minimal in $\gamma^{-1}(\gamma(V))$.
\end{cor}
\begin{proof}
Let $P:=\gamma(V)$. If $V$ is not minimal, then $V/\mathfrak{m}_V$ is not minimal in $\Zar(V/\mathfrak{m}_V|D_P/PD_P)$; hence, the extension $D_P/PD_P\subseteq V/\mathfrak{m}_V$ cannot be algebraic, against Proposition \ref{prop:estresfield-trasc}.
\end{proof}

\section{Dimension 1}\label{sect:dim1}
We now analyze the case where the valuation ring $V$ has dimension $1$; however, the methods we use only work when $V$ is a valuation overring of $D$, i.e., only for the space $\Zar(D)=\Zar(K|D)$, where $K$ is the quotient field of $D$. Unlike in the proof of Theorem \ref{teor:x1xn}, we cannot use \cite[Proposition 3.3]{ZarNoeth2} to extend these results to arbitrary Zariski spaces $\Zar(L|D)$, because that construction changes the dimension of the valuation domains involved. 

The idea of this section is to study the maximal ideals of the finitely generated algebras $D[x_1,\ldots,x_n]$.

\begin{prop}\label{prop:fgalg-maxid-h1}
Let $(D,\mathfrak{m})$ be an integrally closed local domain, and let $T\neq D$ be a finitely generated $D$-algebra contained in the quotient field $K$ of $D$. If $\mathfrak{m}T\neq T$, then no maximal ideal of $T$ above $\mathfrak{m}$ has height $1$.
\end{prop}
\begin{proof}
Let $T:=D[x_1,\ldots,x_n]$; we proceed by induction on $n$.

Suppose $n=1$, and let $x:=x_1$; without loss of generality, $x\notin D$. If $x^{-1}\in D$, then $x\in\mathfrak{m}$, and thus $\mathfrak{m}T=T$, a contradiction. Hence $x,x^{-1}\notin D$. By \cite[Theorem 6]{seiden-dim}, the ideal $\mathfrak{p}:=\mathfrak{m}T$ is prime but not maximal; since every maximal ideal of $T$ above $\mathfrak{m}$ must contain $\mathfrak{p}$, it follows that no such maximal ideal can have height $1$.

Suppose that the claim holds up to $n-1$; let $A:=D[x_1,\ldots,x_{n-1}]$, so that $T=A[x_n]$; without loss of generality, $A\neq D$ and $x_n\notin A$. Let $M$ be a maximal ideal of $T$ above $\mathfrak{m}$. If $x_n$ is integral over $A$, then $T$ is integral over $A$, and thus the height of $M$ is equal to the height of $M\cap A$, which is not equal to $1$ by induction.

Suppose that $x_n$ is not integral over $A$. Let $A'$ be the integral closure of $A$; then, $T\subseteq A'[x_n]$ is an integral extension, and since $x_n$ is not integral over $A$ it follows that $A'\subsetneq A'[x_n]$. Take a maximal ideal $M'$ of $A'[x_n]$ above $M$. Let $N:=M'\cap A'$; then, $N$ is a nonzero prime ideal of $A'$, and thus $A'':=(A')_N$ is a local integrally closed domain with maximal ideal $N(A')_N\neq(0)$. Then, the ring $A''[x_n]$ is the quotient ring of $A'[x_n]$ with respect to the multiplicatively closed set $A'[x_n]\setminus N$, the set $M'':=M'A''[x_n]$ is a maximal ideal, and $N(A')_N\subseteq M''$. Applying the case $n=1$ to $A''$ and $A''[x_n]$, it follows that the height of $M''$ is not $1$; since the height of $M''$ is the same of the height of $M'$ and of $M$, it follows that the height of $M$ is not $1$, as claimed.
\end{proof}

\begin{teor}\label{teor:ic-dim1}
Let $D$ be an integral domain, and let $V\in\Zar(D)$ be a valuation overring of dimension $1$. Then, $V$ is isolated in $\Zar(D)^\cons$ if and only if $V$ is a localization of $\overline{D}$ and its center on $\overline{D}$ is isolated in $\Spec(\overline{D})^\cons$.
\end{teor}
\begin{proof}
Since $\Zar(D)=\Zar(\overline{D})$, we can suppose without loss of generality that $D$ is integrally closed.

If the two conditions hold, then $V$ is isolated by Proposition \ref{prop:essential}.

Suppose that $V$ is isolated in $\Zar(D)^\cons$. Let $P$ be the center of $V$ on $D$, and suppose that $V\neq D_P$. Since $V$ is also isolated in $\Zar(D_P)^\cons$, by Theorem \ref{teor:x1xn} there are $x_1,\ldots,x_n\in K\setminus D_P$ such that $V$ is the integral closure of $D_P[x_1,\ldots,x_n]_M$, where $M$ is a maximal ideal of $D_P[x_1,\ldots,x_n]$. However, $\mathfrak{m}_V\cap D_P[x_1,\ldots,x_n]=M$, and thus $M\cap D_P=PD_P$, so that $PD_P\cdot D_P[x_1,\ldots,x_n]\neq D_P[x_1,\ldots,x_n]$; by Proposition \ref{prop:fgalg-maxid-h1}, $M$ cannot have height $1$. However, the dimension of the integral closure of $D_P[x_1,\ldots,x_n]_M$ is exactly the height of $M$; hence, this contradicts the fact that $V$ has dimension $1$. Thus, $V=D_P$. The fact that $P$ is isolated in $\Zar(D)^\cons$ now follows from Proposition \ref{prop:essential}.
\end{proof}

\begin{cor}
Let $D$ be an integral domain, and let $V\in\Zar(D)$ be a minimal valuation overring of $D$. If $\dim(V)=1$ and $V$ is isolated in $\Zar(D)^\cons$, then the center of $V$ on $D$ has height $1$.
\end{cor}
\begin{proof}
The claim is a direct consequence of Theorem \ref{teor:ic-dim1}.
\end{proof}

Theorem \ref{teor:ic-dim1} does not work when $V$ has dimension 2 or more, as the next example shows.
\begin{ex}
Let $F$ be a field, take two independent indeterminates $X$ and $Y$, and consider $D:=F+XF(Y)[[X]]$, i.e., $D$ is the ring of all power series with coefficients in $F(Y)$ such that the $0$-degree coefficient belongs to $F$. Then, $D$ is a one-dimensional local integrally closed domain (its maximal ideal is $XF(Y)[[X]]$), and its valuation overrings are its quotient field, $F(Y)[[X]]$ and the rings in the form $W+XF(Y)[[X]]$, where $W$ belongs to $\Zar(F(Y)|F)\setminus\{F(Y)\}$, i.e., $W$ is either $F[Y]_{(f)}$ for some irreducible polynomial $f\in F[Y]$) or $W=F[Y^{-1}]_{(Y^{-1})}$.

Each of these $W+XF(Y)[[X]]$ is isolated in $\Zar(D)^\cons$, since each $W$ is isolated in $\Zar(F(Y)|F)$ (this follows, for example, by applying Theorem \ref{teor:noeth} below to $F[Y]$ or to $F[Y^{-1}]$). However, since every $W+XF(Y)[[X]]$ has dimension $2$, it can't be a localization of $D=\overline{D}$.
\end{ex}

\section{The Noetherian case}\label{sect:noeth}
In this section, we want to characterize the isolated points of $\Zar(D)^\cons$ when $D$ is a Noetherian domain. If $D$ is integrally closed, this is a straightforward consequence of  Theorem \ref{teor:ic-dim1}; to extend it to the non-integrally closed case, we need a few lemmas. (Note that the integral closure of a Noetherian domain is not necessarily Noetherian; see e.g. \cite[Example 5, page 209]{nagata_localrings}.)

\begin{lemma}\label{lemma:intersec}
Let $D$ be an integral domain. Let $P$ be a prime ideal and let $\Delta\subseteq\Spec(D)$. If $P=\bigcap\{Q\mid Q\in\Delta\}$, then $P\in\Cl^\cons(\Delta)$.
\end{lemma}
\begin{proof}
Let $=V(aD)\cap\D(J)$ be a basic subset of $\Spec(D)^\cons$ containing $P$, where $a\in D$ and $J$ is a finitely generated ideal. We claim that $\Omega\cap\Delta\neq\emptyset$.

Indeed, for every $Q\in\Delta$, we have $a\in P\subseteq Q$, and thus $\Delta\subseteq\V(aD)$. Furthermore, since $P\in\D(J)$, we have $J\nsubseteq P$, and thus there must be a $Q\in\Delta$ such that $J\nsubseteq Q$. Hence, $Q\in\D(J)\cap\Delta$ and so $Q\in\Omega\cap\Delta$. Therefore, $\Omega\cap\Delta\neq\emptyset$ and $P\in\Cl^\cons(\Delta)$.
\end{proof}

\begin{lemma}\label{lemma:gup-intersec}
Let $A\subseteq B$ be an integral extension, and let $P\in\Spec(A)$, $Q\in\Spec(B)$ be such that $Q\cap A=P$. If $\bigcap\{P'\in\Spec(A)\mid P'\supsetneq P\}=P$, then $\bigcap\{Q'\in\Spec(B)\mid Q'\supsetneq Q\}=Q$.
\end{lemma}
\begin{proof}
Let $I:=\bigcap\{Q'\in\Spec(B)\mid Q'\supsetneq Q\}$, and suppose $I\neq Q$; then, $Q\subsetneq I$ and  $\V(I)=\V(Q)\setminus\{Q\}$. Consider the canonical map of spectra $\phi:\Spec(B)\longrightarrow\Spec(A)$: then, $\phi$ is closed (with respect to the Zariski topology) \cite[Chapter V, \textsection 2, Remark (2)]{bourbaki_ac}, and thus $\phi(\V(I))$ is closed in $\Spec(A)$.

By the lying over and the going up theorems, every $P'\supsetneq P$ belongs to $\phi(\V(I))$, while $P\notin\phi(\V(I))$; hence, $\phi(\V(I))=\V(P)\setminus\{P\}$. However, the condition $\bigcap\{P'\in\Spec(A)\mid P'\supsetneq P\}=P$ shows that $\V(P)\setminus\{P\}$ is not closed (its closure is $\V(P)$), a contradiction. Hence, $I=Q$, as claimed. 
\end{proof}

\begin{teor}\label{teor:noeth}
Let $D$ be a Noetherian domain, and let $V\in\Zar(D)$; let $P$ be the center of $V$ on $D$. Then, $V$ is isolated in $\Zar(D)^\cons$ if and only if $h(P)\leq 1$ and $\V(P)$ is finite.
\end{teor}
\begin{proof}
Suppose first that $V$ is isolated in $\Zar(D)^\cons$.

If $\dim(V)>1$, then $V$ is not Noetherian. By Theorem \ref{teor:x1xn}, $V$ is the integral closure of $D[x_1,\ldots,x_n]_M$, for some $x_1,\ldots,x_n\in V$ and some maximal ideal $M$. However, $D[x_1,\ldots,x_n]$ is Noetherian, and thus so is $D[x_1,\ldots,x_n]_M$; hence, its integral closure is a Krull domain, which can't be a non-Noetherian valuation domain, a contradiction.

If $\dim(V)=0$, then $V=K$. By Proposition \ref{prop:isolated-dim0}, $D$ must be a Goldman domain; by \cite[Theorem 146]{kaplansky}, $\V(P)$ is finite.

If $\dim(V)=1$, then by Theorem \ref{teor:ic-dim1} $V$ is the localization of $\overline{D}$ at a prime ideal of $Q$ of height $1$; hence, $V$ is an essential prime ideal of $\overline{D}$ and thus $Q$ is isolated in $\Spec(\overline{D})^\cons$ by Proposition \ref{prop:essential}. 

Let $P:=Q\cap D$. If $\V(P)$ is infinite, then $P$ is the intersection of all the prime ideals properly containing it (since $D/P$ is not a Goldman domain); by Lemma \ref{lemma:gup-intersec}, the same property holds for $Q$, and thus by Lemma \ref{lemma:intersec} $Q$ is not isolated in $\Spec(D)^\cons$. This is a contradiction, and thus $\V(P)$ must be finite.

\medskip

Conversely, suppose that the two conditions hold, and let $\V(P)=\{P,Q_1,\ldots,Q_n\}$. For each $i$, let $y_i\in Q_i\setminus P$ and let $x_i:=1/y_i$: then, $A:=D[x_1,\ldots,x_n]$ is a Noetherian domain such that $PA$ is a maximal ideal of $A$ of height $\leq 1$; moreover, since $\mm_V\cap D=P$, each $x_i$ belongs to $V$, and thus $V\in\Zar(A)$ and $V\cap A=PA$. 

The subspace $\Zar(A)=\B(x_1,\ldots,x_n)$ is an open set of $\Zar(D)^\cons$: therefore, all isolated points of $\Zar(A)^\cons$ are also isolated in $\Zar(D)^\cons$.

If $P$ has height $0$, then $A=K=V$ and thus $V$ is isolated. Suppose that $h(P)=1$.

Since $A$ is Noetherian, $\{PA\}=\V(PA)$ is an open subset of $\Spec(A)^\cons$; hence, $\gamma_A^{-1}(PA)$ is an open subset of $\Zar(A)^\cons$, where $\gamma_A:\Zar(A)\longrightarrow\Spec(A)$ is the center map relative to $A$. However, $\gamma_A^{-1}(PA)$ is the set of valuation overrings of $A_{PA}=D_P$ centered on $(PA)A_{PA}=PD_P$; since $P$ has height $1$, $D_P$ has dimension $1$, and thus $\gamma_A^{-1}(PA)$ is in bijective correspondence with the maximal ideals of the integral closure $B$ of $D_P$, which is Noetherian by \cite[Theorem 93]{kaplansky}. The Jacobson radical of $B$ is nonzero (since it contains $P$), and thus $B$ has only finitely many maximal ideal; thus $\gamma_A^{-1}(PA)$ is an open finite set of the Hausdorff space $\Zar(A)^\cons$, and so it is discrete. Since $V\in\gamma_A^{-1}(PA)$, we have that $V$ is isolated in $\Zar(A)^\cons$ and thus in $\Zar(D)^\cons$, as claimed.
\end{proof}

\begin{cor}\label{cor:dim3-perfect}
Let $(D,\mm)$ be a Noetherian local domain of dimension at least $3$. Then, $\Zar(D)^\cons$ is perfect.
\end{cor}
\begin{proof}
Suppose $V$ is isolated in $\Zar(D)^\cons$. By Theorem \ref{teor:noeth}, its center $P$ must have height $1$ and $\V(P)$ must be finite. However, since $P$ has height $1$ and the maximal ideal $\mm$ of $D$ has height at least $3$, there is at least one prime ideal between $P$ and $\mm$, and since $D$ is Noetherian there must be infinitely many of them \cite[Theorem 144]{kaplansky}, a contradiction. Hence no $V$ can be isolated, and $\Zar(D)^\cons$ is perfect.
\end{proof}

We now want to show that, when $D$ is countable, there are few possible topological structures for $\Zar(D)^\cons$. The one-dimensional case is very easy.
\begin{prop}\label{prop:omef-noeth-dim1}
Let $(D,\mathfrak{m})$ and $(D',\mathfrak{m}')$ be two Noetherian local domains of dimension $1$. The following are equivalent:
\begin{enumerate}[(i)]
\item\label{prop:omef-noeth-dim1:Max} $|\Max(\overline{D})|=|\Max(\overline{D'})|$;
\item\label{prop:omef-noeth-dim1:Zar} $\Zar(D)\simeq\Zar(D')$;
\item\label{prop:omef-noeth-dim1:cons} $\Zar(D)^\cons\simeq\Zar(D')^\cons$.
\end{enumerate}
\end{prop}
\begin{proof}
Since $D$ is Noetherian and one-dimensional, $\overline{D}$ is a principal ideal domain with finitely many maximal ideals; hence, $\Zar(D)=\Zar(\overline{D})\simeq\Spec(\overline{D})$, and the homeomorphism holds both in the Zariski and in the constructible topology.

Hence, if $|\Max(\overline{D})|=|\Max(\overline{D'})|$ then $\Spec(\overline{D})\simeq\Spec(\overline{D'})$ and thus $\Zar(D)$ and $\Zar(D')$ are homeomorphic in both the Zariski and the constructible topology. Conversely, if $\Zar(D)\simeq\Zar(D')$ (in any of the two topologies) then in particular they have the same cardinality, which is equal to $|\Max(\overline{D})|+1=|\Max(\overline{D'})|+1$; thus $|\Max(\overline{D})|=|\Max(\overline{D'})|$. The claim is proved.
\end{proof}

For larger dimension, we need to join the previous theorems with the topological characterization of the Cantor set. We isolate a lemma.
\begin{lemma}\label{lemma:metrizz}
Let $D$ be a countable domain. Then, $\Zar(D)^\cons$ is metrizable.
\end{lemma}
\begin{proof}
The space $\Zar(D)^\cons$ is compact and Hausdorff, hence normal \cite[Theorem 17.10]{willard_gentop} and in particular regular. Furthermore, the family of sets $\B(t)$ and $\B(t)^c$ (as $t$ ranges in the quotient field of $D$) form a subbasis of $\Zar(D)^\cons$, and thus $\Zar(D)^\cons$ is second countable. By Urysohn's metrization theorem (see e.g. \cite[Theorem 23.1]{willard_gentop}), $\Zar(D)^\cons$ is metrizable.
\end{proof}

\begin{prop}\label{prop:omef-noeth-dim3}
Let $(D,\mathfrak{m})$ and $(D',\mathfrak{m}')$ be two countable Noetherian local domains of dimension at least $3$. Then, $\Zar(D)^\cons\simeq\Zar(D')^\cons$.
\end{prop}
\begin{proof}
Both $\Zar(D)^\cons$ and $\Zar(D')^\cons$ are totally disconnected (since they are zero-dimensional), compact and perfect (Corollary \ref{cor:dim3-perfect}). By Lemma \ref{lemma:metrizz}, they are also metrizable. 

By \cite[Theorem 30.3]{willard_gentop}, any two spaces with these properties are homeomorphic; hence, $\Zar(D)^\cons\simeq\Zar(D')^\cons$.
\end{proof}

To study the case of dimension $2$, we need two further lemmas.
\begin{lemma}\label{lemma:card-X1}
Let $(D,\mathfrak{m})$ be a local Noetherian domain with $\dim(D)>1$. If $D$ is countable, then $D$ has exactly countably many prime ideals of height $1$.
\end{lemma}
\begin{proof}
By \cite[Theorem 144]{kaplansky}, there are infinitely many prime ideals between $(0)$ and $\mathfrak{m}$, and thus $D$ has infinitely many prime ideals of height $1$.

Moreover, every prime ideal is generated by a finite set, and thus the number of prime ideals of height $1$ is at most equal to the number of finite subsets of $D$. Since $D$ is countable, so is the set of its finite subsets; the claim is proved.
\end{proof}

\begin{lemma}\label{lemma:isolated-compact}
Let $(D,\mathfrak{m})$ be a local Noetherian domain of dimension $2$ with quotient field $K$, and let $X$ be the set of isolated points of $\Zar(D)^\cons$. Then:
\begin{enumerate}[(a)]
\item\label{lemma:isolated-compact:comp} $X$ is nonempty and compact, with respect to the Zariski topology;
\item\label{lemma:isolated-compact:count} if $D$ is countable, then $X$ is countable;
\item\label{lemma:isolated-compact:cl} $\Cl^\cons(X)=X\cup\{K\}$;
\item\label{lemma:isolated-compact:deriv} the only isolated point of $(\Zar(D)\setminus X)^\cons$ is $K$;
\item\label{lemma:isolated-compact:deriv2} $\Zar(D)\setminus (X\cup\{K\})$ is closed and perfect, with respect to the constructible topology.
\end{enumerate} 
\end{lemma}
\begin{proof}
\ref{lemma:isolated-compact:comp} Let $\gamma:\Zar(D)\longrightarrow\Spec(D)$ be the center map, and let $X_1$ be the set of height $1$ prime ideals of $D$. We claim that $V\in X$ if and only if $\gamma(V)\in X_1$: indeed, if $V\in X$ then by Theorem \ref{teor:noeth} the height of $\gamma(V)$ is at most $1$, but it can't be $0$ since $\V((0))$ is infinite. On the other hand, if $\gamma(V)\in X_1$ then $\V(\gamma(V))=\{\gamma(V),\mathfrak{m}\}$ is finite, and thus $V\in X$ again by Theorem \ref{teor:noeth}. Therefore, $X=\gamma^{-1}(X_1)$, and in particular $X$ is nonempty.

Since $D$ is a Noetherian ring, $\Spec(D)$ is a Noetherian space with respect to the Zariski topology (i.e., all its subsets are compact; see \cite[Theorem 12.4.3]{spectralspaces-libro} or \cite[Chapter 6, Exercises 5--8]{atiyah}). Since $\gamma$ is a spectral closed map, it is proper, and thus the counterimage of any compact subset of $\Spec(D)$ is compact; therefore, $X=\gamma^{-1}(X_1)$ is compact with respect to the Zariski topology, as claimed.

\ref{lemma:isolated-compact:count} By Lemma \ref{lemma:card-X1}, $X_1$ is countable; furthermore, $\gamma^{-1}(P)$ is finite for every $P\in X_1$, since it is in bijective correspondence with the set of maximal ideals of the integral closure of $D_P$. Since $X=\gamma^{-1}(X_1)$, it follows that $X$ is countable.

\ref{lemma:isolated-compact:cl} Since $X$ is compact, the set $X^\gen=\{W\in\Zar(D)\mid W\supseteq V$ for some $V\in X\}$ is closed in the inverse topology, and thus in the constructible topology; since every element of $X$ is a one-dimensional valuation ring, furthermore, $X^\gen=X\cup\{K\}$. Hence, $\Cl^\cons(X)\subseteq X\cup\{K\}$.

If they are not equal, then $\Cl^\cons(X)=X$. However, $X$ is infinite (since $X_1$ is infinite, by Lemma \ref{lemma:card-X1}) and discrete (by definition, all its points are isolated) and thus it is not compact with respect to the constructible topology; this is a contradiction, since a closed set of a compact set is compact. Thus, $\Cl^\cons(X)=X\cup\{K\}$, as claimed.

\ref{lemma:isolated-compact:deriv} The set $\Zar(D)\setminus(X\cup\{K\})$ is open, with respect to the constructible topology (by part \ref{lemma:isolated-compact:cl}), and its elements are not isolated in $\Zar(D)^\cons$; therefore, none of its elements can be isolated in $(\Zar(D)\setminus X)^\cons$. On the other hand, let $x\in\mm$, $x\neq 0$: then, $D[x^{-1}]$ is a Noetherian domain of dimension $1$, and its maximal ideals are extensions of prime ideals of $D$ of height $1$. Therefore, if $V\in\Zar(D[x^{-1}])$ has dimension $1$ then the center of $V$ on $D$ has height $1$, and thus it is an isolated point of $\Zar(D)$, i.e., $\B(x^{-1})=\Zar(D[x^{-1}])\subseteq X\cup\{K\}$, and $\B(x^{-1})\cap(\Zar(D)\setminus(X\cup\{K\}))=\{K\}$. Since $\B(x^{-1})$ is open in $\Zar(D)^\cons$, it follows that $K$ is isolated in $(\Zar(D)\setminus X)^\cons$.

\ref{lemma:isolated-compact:deriv2} is a direct consequence of \ref{lemma:isolated-compact:deriv}.
\end{proof}

Note that the set $X$ of the previous proposition is \emph{not} compact with respect to the constructible topology, as it is discrete and infinite.

\begin{prop}\label{prop:omef-noeth-dim2}
Let $(D,\mathfrak{m})$ and $(D',\mathfrak{m}')$ be two countable Noetherian local domains of dimension $2$. Then, $\Zar(D)^\cons\simeq\Zar(D')^\cons$.
\end{prop}
\begin{proof}
Denote by $K,K'$ the quotient fields of $D$ and $D'$, respectively.

Let $X$ be the set of isolated points of $\Zar(D)^\cons$, and let $C:=\Zar(D)\setminus(X\cup\{K\})$: then, $C$ is closed in $\Zar(D)^\cons$. Define in the same way $X'$ and $C'$ inside $\Zar(D')$; then, $C'$ is closed.

As in the proof of Proposition \ref{prop:omef-noeth-dim3}, by Lemma \ref{lemma:isolated-compact}\ref{lemma:isolated-compact:deriv2} $C^\cons$ and $(C')^\cons$ are totally disconnected, perfect, compact and metrizable (with respect to the constructible topology), and thus they are homeomorphic. Let $\phi_C:C^\cons\longrightarrow (C')^\cons$ be a homeomorphism.

The set $X$ is discrete and countable, and the unique nonisolated point of $X\cup\{K\}$ is $K$; since the same holds for $X'$ and $K'$,  any bijection $X\longrightarrow X'$ extends to a homeomorphism $\phi_X:(X\cup\{K\})^\cons\longrightarrow (X'\cup\{K'\})^\cons$ by setting $\phi_X(K)=K'$. Define
\begin{equation*}
\begin{aligned}
\phi\colon\Zar(D)^\cons & \longrightarrow\Zar(D')^\cons,\\
V & \longmapsto\begin{cases}\phi_C(V) & \text{if~}V\in C,\\
\phi_X(V) & \text{if~}V\in X\cup\{K\}.
\end{cases}
\end{aligned}
\end{equation*}
By construction, $\phi$ is bijective, and $\phi$ is a homeomorphism when restricted to $C$ and to $X\cup\{K\}$. Since these two sets are closed, by \cite[Theorem 7.6]{willard_gentop} $\phi$ is a homeomorphism. In particular, $\Zar(D)^\cons\simeq\Zar(D')^\cons$.
\end{proof}

We summarize the previous results in the following theorem.
\begin{teor}\label{teor:Zarsimeq}
Let $(D,\mathfrak{m})$ and $(D',\mathfrak{m}')$ be two countable Noetherian local domains. Then, $\Zar(D)^\cons\simeq\Zar(D')^\cons$ if and only if one of the following conditions hold:
\begin{enumerate}[(a)]
\item $\dim(D)=\dim(D')=1$ and $|\Max(\overline{D})|=|\Max(\overline{D'})|$;
\item $\dim(D)=\dim(D')=2$;
\item $\dim(D)\geq 3$ and $\dim(D')\geq 3$.
\end{enumerate}
\end{teor}
\begin{proof}
If $D$ and $D'$ satisfy one of the conditions, then $\Zar(D)^\cons\simeq\Zar(D')^\cons$ by, respectively, Proposition \ref{prop:omef-noeth-dim1}, Proposition \ref{prop:omef-noeth-dim2} and Proposition \ref{prop:omef-noeth-dim3}.

Suppose now that $\Zar(D)^\cons\simeq\Zar(D')^\cons$. 

If $\dim(D)=1$, then $\Zar(D)$ is finite, and thus so must be $\Zar(D')$; hence, $\dim(D')=1$, and $|\Max(\overline{D})|=|\Max(\overline{D'})|$ by Proposition \ref{prop:omef-noeth-dim1}.

Suppose $\dim(D),\dim(D')\geq 2$. By Corollary \ref{cor:dim3-perfect} and Lemma \ref{lemma:card-X1}, $\Zar(D)^\cons$ has isolated points if and only if $\dim(D)=2$; therefore $\dim(D)=2$ if and only if $\dim(D')=2$, and $\dim(D)\geq 3$ if and only if $\dim(D')\geq 3$. The claim is proved.
\end{proof}

\section{When $D$ is a field}\label{sect:Dfield}
In this section we analyze the isolated points Zariski space $\Zar(L|D)^\cons$ when $D=K$ is a field. Note that, if $L$ is algebraic over $K$, then $\Zar(L|K)$ is just a point ($L$ itself); thus, the only interesting case is when $\trdeg(L/K)\geq 1$.

We start by connecting the isolated points of $\Zar(L|D)^\cons$ and of $\Zar(L'|D)^\cons$, where $L'\subseteq L$ is an algebraic extension.
\begin{prop}\label{prop:algext}
Let $V$ be a valuation domain, and $L'\subseteq L$ be an algebraic extension such that $V\subseteq L'$. Let $\rho:\Zar(L|V)\longrightarrow\Zar(L'|V)$ be the restriction map, and let $\mathcal{X}\subseteq\Zar(L|V)$ be a subset such that $\rho^{-1}(\rho(\mathcal{X}))=\mathcal{X}$. Then, the following hold.
\begin{enumerate}[(a)]
\item\label{prop:algext:isolated} If $W$ is isolated in $\mathcal{X}^\cons$, then $\rho(W)$ is isolated in $\rho(\mathcal{X})^\cons$.
\item\label{prop:algext:perfect} If $\rho(\mathcal{X})$ is perfect and $|\rho(\mathcal{X})|>1$, then $\mathcal{X}$ is perfect.
\end{enumerate}
In particular, the previous statements apply to $\mathcal{X}=\Zar(L|V)$ and $\mathcal{X}=\est(L|V)$.
\end{prop}
\begin{proof}
\ref{prop:algext:isolated} Let $W$ be an isolated point of $\mathcal{X}^\cons$, and let $W':=W\cap L'=\rho(W)$.

Suppose first that $L$ is finite and normal over $L'$. Let $G$ be the group of $L'$-automorphisms of $L$: then, every $\sigma\in G$ is  continuous when seen as a map from $\Zar(L|V)^\cons$ to itself. Moreover, $\rho(\sigma(Z))=\rho(Z)$ for every $Z\in\Zar(L|V)$, and thus $\sigma$ restricts to a self-homeomorphism of $\mathcal{X}$.

Since $G$ acts transitively on $\rho^{-1}(W')$ (see e.g. \cite[Corollary 20.2]{gilmer}) and $W\in\rho^{-1}(W')$ is isolated, all points of $\rho^{-1}(W')$ are isolated in $\mathcal{X}$; hence, $\rho^{-1}(W')$ is open in $\mathcal{X}^\cons$. Since $\rho:\Zar(L|V)\longrightarrow\Zar(L'|V)$ is a closed map (with respect to the constructible topology), it is also closed when seen as a map $\mathcal{X}\longrightarrow\rho(\mathcal{X})$; therefore, $\rho(\mathcal{X}\setminus\rho^{-1}(W'))=\rho(\mathcal{X})\setminus\{W'\}$ is closed in $\rho(\mathcal{X})$, with respect to the constructible topology, and thus $W'$ is an isolated point of $\rho(\mathcal{X})^\cons$, as claimed.

\smallskip

Suppose now that $L$ is finite over $L'$, and let $F$ be the normal closure of $L$. Let $\rho_0:\Zar(F|V)\longrightarrow\Zar(L|V)$ be the restriction map. Since $W$ is isolated in $\mathcal{X}$, the set $\rho_0^{-1}(W)$ is open in $\rho_0^{-1}(\mathcal{X})^\cons$; moreover, $\rho_0^{-1}(W)$ is finite since $[F:L]<\infty$. Therefore, $\rho_0^{-1}(W)$ is a discrete subspace of $\rho_0^{-1}(\mathcal{X})^\cons$, and in particular each $Z\in\rho_0^{-1}(W)$ is isolated. Applying the previous part of the proof to the extension $L'\subseteq F$ and to any such $Z$, we obtain that $Z\cap L'=W\cap L'=\rho(W)$ is isolated, as claimed.

\smallskip

Suppose now that $L'\subseteq L$ is arbitrary. Since $W$ is isolated in $\mathcal{X}$, there are $x_1,\ldots,x_n,y_1,\ldots,y_m\in L$ such that $\{W\}=\B(x_1,\ldots,x_n)\cap\B(y_1)^c\cap\cdots\cap\B(y_m)^c\cap\mathcal{X}$. Let $F:=L(x_1,\ldots,x_n,y_1,\ldots,y_m)$: then, $W\cap F$ is isolated in $\{Z\cap F\mid Z\in\mathcal{X}\}$. Since $[F:L']<\infty$, we can apply the previous part of the proof, obtaining that $W\cap F\cap L'=W\cap L'=\rho(W)$ is isolated in $\rho(\mathcal{X})^\cons$, as claimed.

\medskip

\ref{prop:algext:perfect} Suppose that $\mathcal{X}$ is not perfect: then, there is a $W\in\mathcal{X}$ that is isolated. By the previous part of the proof, it would follow that $W\cap L'$ is isolated in $\rho(X)^\cons$. Since $\rho(\mathcal{X})$ has more than one point, this is impossible, and so $\mathcal{X}$ is perfect.

\medskip

The ``in particular'' statement follows from the fact that $\Zar(L|V)$ and $\est(L|V)$ satisfy the hypothesis on $\mathcal{X}$.
\end{proof}

\begin{cor}\label{cor:algext-ZarExt}
Let $V$ be a valuation domain and $L'\subseteq L$ be an algebraic extension; suppose that $V\subseteq L'$ and that $L'$ is transcendental over the quotient field of $V$. If $\Zar(L'|V)^\cons$ (respectively, $\est(L'|V)^\cons$) is perfect, then $\Zar(L|V)^\cons$ (resp., $\est(L|V)^\cons$) is perfect.
\end{cor}
\begin{proof}
It is enough to apply Proposition \ref{prop:algext}\ref{prop:algext:perfect} to $\mathcal{X}=\Zar(L|V)$ or $\mathcal{X}=\est(L|V)$, using the hypothesis that $L'$ is transcendental over the quotient field of $V$ to guarantee that $|\Zar(L|V)|>1$ and $|\est(L|V)^\cons|>1$.
\end{proof}

The following result completely settles the problem of finding the isolated points when $\trdeg(L/K)\geq 2$, generalizing \cite[Theorem 4.45]{preordered-groups} and solving the authors' Conjecture A (in an even more general formulation). Note that the first case in the proof is exactly \cite[Theorem 4.45]{preordered-groups}, but we give a new proof of it using Theorem \ref{teor:noeth}.
\begin{teor}\label{teor:trdeg2-field}
Let $K\subseteq L$ be a field extension with $\trdeg(L/K)\geq 2$. Then, $\Zar(L|K)^\cons$ is perfect.
\end{teor}
\begin{proof}
Suppose first that $L=K(x_1,\ldots,x_n)$ is a finitely generated purely transcendental extension of $K$, with transcendence basis $x_1,\ldots,x_n$. Suppose there exists an isolated point $W$ of $\Zar(L|K)^\cons$. By Proposition \ref{prop:isolated-dim0}, $W\neq L$.

For each $i$, at least one of $x_i$ and $x_i^{-1}$ belongs to $W$; let it be $t_i$. Then, $W\in\Zar(K[t_1,\ldots,t_n])$, and so $W$ is isolated in $\Zar(K[t_1,\ldots,t_n])^\cons$. Let $P$ be the center of $W$ on $K[t_1,\ldots,t_n]$; since $K[t_1,\ldots,t_n]$ is Noetherian, by Theorem \ref{teor:noeth} $P$ has height $1$ and $\V(P)$ is finite.

Since $K[t_1,\ldots,t_n]$ is isomorphic to a polynomial ring, every maximal ideal of $K[t_1,\ldots,t_n]$ has height $n>1$ \cite[Section 3.2, Exercise 3]{kaplansky}, and thus $P$ is not maximal. However, $K[t_1,\ldots,t_n]$ is an Hilbert ring, and thus every non-maximal prime ideal is the intersection of the maximal ideals containing it \cite[Theorem 147]{kaplansky}; in particular, this happens for $P$, and thus $\V(P)$ must be infinite. This is a contradiction, and so $\Zar(L|K)^\cons$ is perfect.

\medskip

Suppose now that $L$ has finite transcendence degree over $K$, let $x_1,\ldots,x_n$ be a transcendence basis of $L$ and let $L':=K(x_1,\ldots,x_n)$. By the previous part of the proof, $\Zar(L'|K)^\cons$ is perfect; since $L'\subseteq L$ is algebraic, by Corollary \ref{cor:algext-ZarExt} also $\Zar(L|K)^\cons$ is perfect.

\medskip

Take now any extension $L$ of $K$, and suppose that $W$ is an isolated point of $\Zar(L|K)^\cons$. Then, there are $x_1,\ldots,x_n,y_1,\ldots,y_m\in L$ such that $\{W\}=\B(x_1,\ldots,x_n)\cap\B(y_1)^c\cap\cdots\cap\B(y_m)^c$. Take two elements $a,b\in L$ that are algebraically independent over $K$, and let $L':=K(a,b,x_1,\ldots,x_n,y_1,\ldots,y_m)$: then, $2\leq\trdeg(L'/K)<\infty$. Set $V:=W\cap L'$: then, $\{V\}=\B^{L'}(x_1,\ldots,x_n)\cap\B^{L'}(y_1)^c\cap\cdots\cap\B^{L'}(y_m)^c$, and thus $V$ is isolated in $\Zar(L'|V)^\cons$. However, by the previous part of the proof, $\Zar(L'|V)^\cons$ is perfect, a contradiction. Hence $\Zar(L|K)^\cons$ is perfect.
\end{proof}

When the transcendence degree of $L$ over $K$ is $1$, the picture is very different, because it may even happen that all elements of $\Zar(L|K)^\cons$ (except $L$ itself) are isolated. Compare the next results with \cite[Corollary 5.5(a)]{ZarNoeth} and \cite[Proposition 4.2]{ZarNoeth2}.
\begin{prop}\label{prop:isolated-K(X)K}
Let $K$ be a field. Then all points of $\Zar(K(X)|K)$, except $K(X)$, are isolated with respect to the constructible topology.
\end{prop}
\begin{proof}
The points of $\Zar(K(X)|K)$ are $K(X)$, $K[X]_{(X^{-1})}$ and the rings $K[X]_{(f(X))}$, where $f(X)$ is an irreducible polynomial of $K[X]$. The first one is not isolated by Proposition \ref{prop:isolated-dim0}; on the other hand, $\{K[X]_{(f(X))}\}=\B(f(X)^{-1})^c$ and $\{K[X^{-1}]_{(X^{-1})}\}=\B(X)^c$, and thus these domains are isolated, as claimed.
\end{proof}

\begin{lemma}\label{lemma:isolato-extfin}
Let $D$ be an integral domain with quotient field $K$, and let $L'\subseteq L$ be two extensions of $K$. Let $V\in\Zar(L'|D)$. If $V$ is isolated in $\Zar(L'|D)^\cons$ and $\est(L|V)$ is finite, then every $W\in\est(L|V)$ is isolated in $\Zar(L|D)^\cons$.
\end{lemma}
\begin{proof}
Let $\rho:\Zar(L|D)\longrightarrow\Zar(L'|D)$ be the restriction map. Then, $\est(L|V)=\rho^{-1}(V)$ is open in $\Zar(L'|D)^\cons$ since $V$ is isolated. Moreover, $\est(L|V)$ is finite by hypothesis, and, since the constructible topology is Hausdorff, all its points are isolated in $\Zar(L|D)^\cons$.
\end{proof}

\begin{prop}\label{prop:LK-trdeg1-isolated}
Let $K$ be a field, and let $L$ be an extension of $K$ such that $\trdeg(L/K)=1$. Let $V\in\Zar(L|K)$, $V\neq L$. Then, the following are equivalent:
\begin{enumerate}[(i)]
\item\label{prop:LK-trdeg1-isolated:isol} $V$ is isolated in $\Zar(L|K)^\cons$;
\item\label{prop:LK-trdeg1-isolated:unique} there is a finitely generated extension $L'$ of $K$ such that $L'\subseteq L$ and $\est(L|V\cap L')=\{V\}$;
\item\label{prop:LK-trdeg1-isolated:finite} there is a finitely generated extension $L'$ of $K$ such that $L'\subseteq L$ and $\est(L|V\cap L')$ is finite.
\end{enumerate}
\end{prop}
\begin{proof}
\ref{prop:LK-trdeg1-isolated:isol} $\Longrightarrow$ \ref{prop:LK-trdeg1-isolated:unique} Since $V$ is isolated, we have $\{V\}=\Omega:=\B(x_1,\ldots,x_n)\cap\B(y_1)^c\cap\cdots\cap\B(y_m)^c$, for some $x_1,\ldots,x_n,y_1,\ldots,y_m\in L$. Let $L':=K(x_1,\ldots,x_n,y_1,\ldots,y_m)$: then, every extension of $V\cap L'$ to $L$ belongs to $\Omega$, and thus it is equal to $V$. Hence $L'$ is the required field.

\ref{prop:LK-trdeg1-isolated:unique} $\Longrightarrow$ \ref{prop:LK-trdeg1-isolated:finite} is obvious.

\ref{prop:LK-trdeg1-isolated:finite} $\Longrightarrow$ \ref{prop:LK-trdeg1-isolated:isol} Since $\est(L|V\cap L')$ is finite, $L'\subseteq L$ must be algebraic and so $K\subseteq L'$ is transcendental; take any $X\in L'$ that is transcendental over $K$. Since $K\subseteq L'$ is finitely generated, $K(X)\subseteq L'$ must be a finite extension.

Since $V\neq L$, we have $V\cap K(X)\neq K(X)$; by Proposition \ref{prop:isolated-K(X)K}, $V\cap K(X)$ is isolated in $\Zar(K(X)|K)$. Moreover, since $K(X)\subseteq L'$ is a finite extension, $\est(K(X)|V\cap K(X))$ is finite; by Lemma \ref{lemma:isolato-extfin}, all its points (in particular, $V$) are isolated in $\Zar(L|K)^\cons$.
\end{proof}

\begin{prop}\label{prop:LK-trdeg1}
Let $K$ be a field, and let $L$ be an extension of $K$ such that $\trdeg(L/K)=1$. Let $\mathcal{X}:=\Zar(L|K)\setminus\{L\}$. Then, the following are equivalent:
\begin{enumerate}[(i)]
\item\label{prop:LK-trdeg1:isolated} all points of $\mathcal{X}$ are isolated in $\Zar(L|K)^\cons$;
\item\label{prop:LK-trdeg1:forall} for every $X\in L$, transcendental over $K$, the set $\est(L|V)$ is finite for every $V\in\Zar(K(X)|K)$;
\item\label{prop:LK-trdeg1:exists} there is an $X\in L$, transcendental over $K$, such that the set $\est(L|V)$ is finite. for every $V\in\Zar(K(X)|K)$.
\end{enumerate}
\end{prop}
\begin{proof}
\ref{prop:LK-trdeg1:isolated} $\Longrightarrow$ \ref{prop:LK-trdeg1:forall} Take any $X\in L$ that is transcendental over $K$, and let $V\in\Zar(K(X)|K)$. The space $\est(L|V)$ is closed in $\Zar(L|V)^\cons$, and thus it is compact. Since all its points are isolated, it is also discrete; hence, $\est(L|V)$ is finite.

\ref{prop:LK-trdeg1:forall} $\Longrightarrow$ \ref{prop:LK-trdeg1:exists} is obvious.

\ref{prop:LK-trdeg1:exists} $\Longrightarrow$ \ref{prop:LK-trdeg1:isolated} Apply Proposition \ref{prop:LK-trdeg1-isolated}, \ref{prop:LK-trdeg1-isolated:finite} $\Longrightarrow$ \ref{prop:LK-trdeg1-isolated:isol} with $L'=K(X)$ to each $V\in\mathcal{X}$.
\end{proof}

\begin{cor}
Let $K$ be a field and let $L$ be a finitely generated extension of $K$ such that $\trdeg(L/K)=1$. Then, all points of $\Zar(L|K)\setminus\{L\}$ are isolated in $\Zar(L|K)^\cons$.
\end{cor}
\begin{proof}
It is enough to apply Proposition \ref{prop:LK-trdeg1}.
\end{proof}

\begin{oss}
Let $K\subseteq L$ be a transcendental extension of degree $1$, and let $V\in\Zar(L|K)$. Let $X\in L$ be transcendental over $K$. By Proposition \ref{prop:LK-trdeg1-isolated}, if $\est(K(X)|V\cap K(X))$ is finite, then $V$ is isolated in $\Zar(L|K)^\cons$; however, unlike in Proposition \ref{prop:LK-trdeg1}, the converse does not hold, i.e., $\est(K(X)|V\cap K(X))$ may be infinite even if $V$ is isolated.

For example, let $W=K[X]_{(X)}$ (more generally, we can take any $W\in\Zar(K(X)|K)$, $W\neq K(X)$). Since $W$ is a discrete valuation ring, using \cite{krull-estval} (see also \cite[Section 3]{gilmer-Pruf-Int}), it is possible to construct a chain $K(X)\subset F_0\subset F_1\subset\cdots$ of extensions of $K(X)$ such that:
\begin{itemize}
\item the extensions $K(X)\subset F_0$ and $F_i\subset F_{i+1}$ are finite, for each $i>0$;
\item $W$ has two extensions to $F_0$, say $W_1$ and $W_2$;
\item $W_1$ has only one extension to $F_i$, for each $i>0$;
\item if $W'$ is an extension of $W_2$ to $F_i$, then $W'$ has more than one extension to $F_{i+1}$.
\end{itemize}
Let $L:=\bigcup_{i\geq 0}F_i$. Then, $W_1$ has a unique extension $V$ to $L$, while $W_2$ has infinitely many extensions; in particular, the set $\mathcal{X}$ of extensions of $W$ to $L$ is infinite. Let $y\in W_1\setminus W_2$: then, $\B(y)\cap\mathcal{X}=\{V\}$, and thus $\B(y)\cap\B(X^{-1})^c=\{V\}$. Hence, $V$ is isolated in $\Zar(L|K)^\cons$, despite $V\cap K(X)=W$ having infinitely many extensions to $L$.

The reason why the proof of Proposition \ref{prop:LK-trdeg1} fails in this context is that we are not requiring the \emph{other} extensions of $W$ to $L$ to be isolated.
\end{oss}

\section{Extensions of valuations}\label{sect:extval}
In this section, we extend the results of the previous section from the case where $D=K$ is a field to the case where $D=V$ is a valuation domain. In particular, we want to study the set $\est(L|V)$ of extensions of $V$ to $L$.

The most important case is when $L=K(X)$ is the field of rational functions. If $V$ is a valuation domain with quotient field $K$ and $s\in K$, we set
\begin{equation*}
V_s:=\{\phi\in K(X)\mid \phi(s)\in V\}.
\end{equation*}
Then, $V_s$ is an extension of $V$ to $K(X)$, and it is possible to analyze quite thoroughly its algebraic properties (see for example \cite[Proposition 2.2]{PerTransc} for a description when $V$ has dimension $1$).

The following lemma is a partial generalization of \cite[Theorem 3.2]{PerTransc}, of which we follow the proof.
\begin{lemma}\label{lemma:Us-coniugati}
Let $V$ be a valuation domain with quotient field $K$, and let $U$ be an extension of $V$ to the algebraic closure $\overline{K}$. Let $s,t\in\overline{K}$. Then, $U_s\cap K(X)=U_t\cap K(X)$ if and only if $s$ and $t$ are conjugated over $K$.
\end{lemma}
\begin{proof}
If $s,t$ are conjugated, there is a $K$-automorphism $\sigma$ of $\overline{K}$ sending $s$ to $t$. Setting $\widetilde{\sigma}(\sum_ia_iX^i):=\sum_i\sigma(a_i)X^i$, we can extend $\sigma$ to a $K(X)$-automorphism $\widetilde{\sigma}$ of $\overline{K}(X)$ such that $\widetilde{\sigma}(\phi)(t)=\sigma(\phi(s))$ for every $\phi\in\overline{K}(X)$; in particular, if $\phi\in K(X)$ then $\widetilde{\sigma}(\phi)=\phi$ and thus $\phi(s)\in V$ if and only if $\phi(t)\in V$, i.e., $\phi\in U_s\cap K(X)$ if and only if $\phi\in U_t\cap K(X)$. Therefore, $U_s\cap K(X)=U_t\cap K(X)$.

Conversely, suppose that $s$ and $t$ are not conjugate, and let $p(X)$ be the minimal polynomial of $s$ over $K$: then, $p(t)\neq 0$, and thus there is a $c\in K$ such that $v(c)>u(p(t))$ (where $v$ and $u$ are, respectively, the valuations with respect to $V$ and $U$ and $u|_K=v$). Then, $q(X):=\frac{p(X)}{c}\in K(X)$ belongs to $U_s$ (since $q(s)=0\in V$) but not to $U_t$ (since $u(q(t))=u(p(t))-v(c)<0$). Hence, $U_s\cap K(X)\neq U_t\cap K(X)$, as claimed.
\end{proof}

\begin{teor}\label{teor:ext-K(X)}
Let $V$ be a valuation domain that is not a field. Then, $\est(K(X)|V)^\cons$ is perfect.
\end{teor}
\begin{proof}
Suppose first that $K$ is algebraically closed. By \cite[Theorem 7.2]{fundstatz}, for all extensions $W$ of $V$ to $K(X)$ there is a sequence $E=\{s_\nu\}_{\nu\in\Lambda}$ (where $\Lambda$ is a well-ordered set without maximum) such that
\begin{equation*}
W=V_E=\{\phi\in K(X)\mid \phi(s_\nu)\in V\text{~for all large~}\nu\}
\end{equation*}
and $W\neq V_{s_\nu}$ for every $\nu$. In particular, the elements $\phi(s_\nu)$ are either eventually in $V$ or eventually out of $V$ (by \cite[Proposition 3.2]{fundstatz}; see also the proof of Theorem 3.4 therein). Take $\psi\in K(X)$: then, if $W\in\B(\psi)$ then it must be $\psi(s_\nu)\in V$ eventually, and thus $\B(\psi)$ contains $V_{s_\nu}$ for all large $\nu$; on the other hand, if $W\in\B(\psi)^c$ then $\phi(s_\nu)\notin V$ eventually, and thus $\B(\psi)^c$ contains $V_{s_\nu}$ for all large $\nu$.

Let now $\Omega:=\B(\psi_1,\ldots,\psi_n)\cap\B(\phi_1)^c\cap\cdots\cap\B(\phi_m)^c\cap\est(K(X)|V)$ be a basic open set of $\est(K(X)|V)^\cons$ containing $W$. For every $i$, there is an index $N_i$ such that $\psi_i(s_\nu)\in V$ for all $i\geq N_i$; likewise, for every $j$ there is a $M_j$ such that $\phi_j(s_\nu)\notin V$ for all $\nu\geq M_j$. Therefore, for every $\nu\geq\max\{N_1,\ldots,N_n,M_1,\ldots,M_m\}$, we have $V_{s_\nu}\in\Omega$; hence, $W$ belongs to the closure of $\{V_{s_\nu}\}_{\nu\in\Lambda}\subseteq\est(K(X)|V)$, with respect to the constructible topology. Therefore, $W$ is not isolated in $\est(K(X)|V)^\cons$ and, since $W$ was arbitrary, $\est(K(X)|V)^\cons$ is perfect.

\medskip

Suppose now that $K$ is any field. Let $W\in\est(K(X)|V)$, and suppose that $W$ is isolated in $\est(K(X)|V)^\cons$. Let $\rho:\est(\overline{K}(X)|V)\longrightarrow\est(K(X)|V)$ be the restriction map. Since $W$ is isolated and $\rho$ is continuous, $\rho^{-1}(W)$ is open. Let $W'\in\rho^{-1}(W)$ and let $U:=W'\cap\overline{K}$: then, $U$ is an extension of $V$ to $\overline{K}$.

By the previous part of the proof, for every open neighborhood $\Omega$ of $W'$ there is an $s\in\overline{K}$ such that $U_s\neq W$ and $U_s\in\Omega$; since $\rho^{-1}(W)$ is open, it follows that for every such $\Omega$ there is a $U_s\in\rho^{-1}(W)$ with these properties. Therefore, the set $\Delta:=\{U_s\in\rho^{-1}(W)\mid s\in\overline{K}\}$ is dense in $\rho^{-1}(W)$. Since $U_s\cap K(X)=U_t\cap K(X)=W$ for every $U_s,U_t\in\Delta$, by Lemma \ref{lemma:Us-coniugati} $\Delta$ is finite; since $\est(K(X)|V)^\cons$ is Hausdorff, it follows that $\Delta=\rho^{-1}(W)$, and in particular $\rho^{-1}(W)$ is finite. Hence, all its points are isolated. However, this contradicts the fact that  $\est(\overline{K}(X)|V)^\cons$ is perfect; thus, also $\est(K(X)|V)^\cons$ must be perfect.
\end{proof}

The previous theorem allows to determine the isolated points of $\Zar(K(X)|D)^\cons$ for every integral domain $D$.

\begin{prop}
Let $D$ be an integral domain that is not a field, and let $J$ be the intersection of the nonzero prime ideals of $D$.
\begin{enumerate}[(i)]
\item If $J=(0)$, then $\Zar(K(X)|D)^\cons$ is perfect.
\item If $J\neq(0)$, then the only isolated points of $\Zar(K(X)|D)^\cons$ are $K[X]_{(f(X))}$ (where $f(X)$ is an irreducible polynomial of $K[X]$) and $K[X]_{(X^{-1})}$.
\end{enumerate}
\end{prop}
\begin{proof}
Let $W\in\Zar(K(X)|D)$. If $W\cap K\neq K$, then $W\in\est(K(X)|V)$, which is perfect (when endowed with the constructible topology) by Theorem \ref{teor:ext-K(X)}. In particular, $W$ is not isolated in $\Zar(K(X)|D)^\cons$.

Suppose that $W\cap K=K$. If $W=K(X)$, then $W$ is not isolated by Proposition \ref{prop:isolated-dim0}, since $K(X)$ is not algebraic over $K$. Let thus $W\neq K(X)$.

Suppose that $J=(0)$, and let $W$ be an isolated point of $\Zar(K(X)|D)$. Since $K\subseteq W$, we have $\mm_W\cap D=(0)$; by Lemma \ref{lemma:quoziente}, the quotient map of $W$ onto its residue field induces a homeomorphism between the spaces $\Delta:=\{Z\in\Zar(K(X)|D)\mid Z\subseteq W\}$ and $\Zar(W/\mm_W|D)$, where $W$ is sent to $W/\mm_W$. Since $W$ is isolated in $\Zar(K(X)|D)^\cons$, it is also isolated in $\Delta^\cons$, and thus $W/\mm_W$ must be an isolated point of $\Zar(W/\mm_W|D)^\cons$. By Proposition \ref{prop:isolated-dim0}, $D$ must be a Goldman domain, against the hypothesis $J=(0)$. Therefore, $W$ is not isolated and $\Zar(K(X)|D)^\cons$  is perfect.

Suppose now that $J\neq(0)$, and let $j\in J$, $j\neq 0$. Then, $D[j^{-1}]=K$, and thus $\B(j^{-1})=\est(K(X)|K)=\Zar(K(X)|K)$ is a clopen subset of $\Zar(K(X)|D)^\cons$; in particular, $W\in\est(K(X)|K)$ is isolated in $\Zar(K(X)|D)^\cons$ if and only if it is isolated in $\Zar(K(X)|K)^\cons$. The claim now follows from Proposition \ref{prop:isolated-K(X)K}.
\end{proof}

To conclude the paper, we extend Theorem \ref{teor:trdeg2-field} to valuation domains.
\begin{teor}\label{teor:trdeg2-val}
Let $V$ be a valuation domain with quotient field $K$, and let $L$ be a field extension of $K$ such that $\trdeg(L/K)\geq 2$. Then, $\est(L|V)^\cons$ and $\Zar(L|V)^\cons$ are perfect.
\end{teor}
\begin{proof}
We first show that $\est(L|V)^\cons$ is perfect: suppose that is not, and let $W$ be an isolated point.

Suppose that $L=K(x,z_2,\ldots)$ is purely transcendental over $K$, where $x,z_2,\ldots,$ is a transcendence basis. Take an $m\in\mathfrak{m}_V\subseteq\mathfrak{m}_W$: then, at least one of $mx$ and $x^{-1}$ belongs to $\mm_W$. Let $z_1$ be that element. Then, $z_1,z_2,\ldots$ is also a transcendence basis of $L$.

Let $L':=K(z_1,z_3,\ldots,)$ be the extension of $K$ obtained adjoining all the element of this basis except $z_2$. Then, $z_1^{-1}\in L'\setminus W$, and thus $W\cap L'\neq L'$; since, by construction, $L\simeq L'(X)$, by Theorem \ref{teor:ext-K(X)} $\est(L|W\cap L')^\cons$ is perfect. Since $\est(L|W\cap L')\subseteq\est(L|W)$, all the elements of $\est(L|W\cap L')$ (in particular, $W$) are not isolated in $\est(L|V)^\cons$. This is a contradiction, and thus $\est(L|V)^\cons$ is perfect.

Suppose now that $L$ is arbitrary: then, we can find a purely transcendental extension $L'$ of $K$ such that $L'\subseteq L$ is algebraic. By the previous part of the proof, $\est(L'|V)^\cons$ is perfect; by Corollary \ref{cor:algext-ZarExt}, also $\est(L|V)^\cons$ is perfect. Therefore, $\est(L|V)^\cons$ is always perfect.

Finally, $\Zar(L|V)$ is the union of $\est(L|V_0)$, as $V_0$ ranges among the valuation overrings of $V$; since each of these is perfect with respect to the constructible topology (by the previous part of the proof), then also $\Zar(L|V)^\cons$ is perfect, as claimed.
\end{proof}

\begin{cor}
Let $V$ be a valuation domain with quotient field $K$, suppose $V\neq K$, and let $L$ be a transcendental field extension of $K$. Then, $\est(L|V)^\cons$ is perfect.
\end{cor}
\begin{proof}
If $\trdeg(L/K)\geq 2$ the claim follows from Theorem \ref{teor:trdeg2-val}. If $\trdeg(L/K)=1$, let $X\in L$ be transcendental over $K$. By Theorem \ref{teor:ext-K(X)}, $\est(K(X)|V)^\cons$ is perfect; by Corollary \ref{cor:algext-ZarExt}, also $\est(L|V)^\cons$ is perfect.
\end{proof}

\begin{cor}\label{cor:LD-trdeg2}
Let $D$ be an integral domain, and let $L$ be a transcendental extension of the quotient field $K$ of $D$. If $\trdeg(L/K)\geq 2$, then $\Zar(L|D)^\cons$ is perfect.
\end{cor}
\begin{proof}
Any $W\in\Zar(L|D)$ belongs to $\est(L|V)$ for some $V\in\Zar(D)$. By Theorem \ref{teor:trdeg2-val}, all $\est(L|V)^\cons$ are perfect, and thus no $W$ is isolated. Hence, $\Zar(L|D)^\cons$ is perfect.
\end{proof}

\section*{Acknowledgments}
I would like to thank the referee for pointing out several problems in the first version of the paper.

\bibliographystyle{plain}
\bibliography{/bib/articoli,/bib/libri,/bib/miei}

\begin{thebibliography}{10}

\bibitem{atiyah}
M.~F. Atiyah and I.~G. Macdonald.
\newblock {\em Introduction to {C}ommutative {A}lgebra}.
\newblock Addison-Wesley Publishing Co., Reading, Mass.-London-Don Mills, Ont.,
  1969.

\bibitem{bourbaki_ac}
Nicolas Bourbaki.
\newblock {\em Commutative {A}lgebra. {C}hapters 1--7}.
\newblock Elements of Mathematics (Berlin). Springer-Verlag, Berlin, 1989.
\newblock Translated from the French, Reprint of the 1972 edition.

\bibitem{preordered-groups}
Julie Decaup and Guillaume Rond.
\newblock Preordered groups and valued fields.
\newblock {\em preprint, arXiv:1912.03928}, 2019.

\bibitem{spectralspaces-libro}
Max Dickmann, Niels Schwartz, and Marcus Tressl.
\newblock {\em Spectral spaces}, volume~35 of {\em New Mathematical
  Monographs}.
\newblock Cambridge University Press, Cambridge, 2019.

\bibitem{dobbs_fedder_fontana}
David~E. Dobbs, Richard Fedder, and Marco Fontana.
\newblock Abstract {R}iemann surfaces of integral domains and spectral spaces.
\newblock {\em Ann. Mat. Pura Appl. (4)}, 148:101--115, 1987.

\bibitem{fontana_krr-abRs}
David~E. Dobbs and Marco Fontana.
\newblock Kronecker function rings and abstract {R}iemann surfaces.
\newblock {\em J. Algebra}, 99(1):263--274, 1986.

\bibitem{finocchiaro-ultrafiltri}
Carmelo~A. Finocchiaro.
\newblock Spectral spaces and ultrafilters.
\newblock {\em Comm. Algebra}, 42(4):1496--1508, 2014.

\bibitem{fifolo_transactions}
Carmelo~A. Finocchiaro, Marco Fontana, and K.~Alan Loper.
\newblock The constructible topology on spaces of valuation domains.
\newblock {\em Trans. Amer. Math. Soc.}, 365(12):6199--6216, 2013.

\bibitem{spettrali-eab}
Carmelo~A. Finocchiaro, Marco Fontana, and Dario Spirito.
\newblock Spectral spaces of semistar operations.
\newblock {\em J. Pure Appl. Algebra}, 220(8):2897--2913, 2016.

\bibitem{topological-cons}
Carmelo~A. Finocchiaro and Dario Spirito.
\newblock Some topological considerations on semistar operations.
\newblock {\em J. Algebra}, 409:199--218, 2014.

\bibitem{fontana_libro}
Marco Fontana, James~A. Huckaba, and Ira~J. Papick.
\newblock {\em Pr\"ufer {D}omains}, volume 203 of {\em Monographs and Textbooks
  in Pure and Applied Mathematics}.
\newblock Marcel Dekker Inc., New York, 1997.

\bibitem{gilmer}
Robert Gilmer.
\newblock {\em Multiplicative {I}deal {T}heory}.
\newblock Marcel Dekker Inc., New York, 1972.
\newblock Pure and Applied Mathematics, No. 12.

\bibitem{gilmer-Pruf-Int}
Robert Gilmer.
\newblock Pr\"{u}fer domains and rings of integer-valued polynomials.
\newblock {\em J. Algebra}, 129(2):502--517, 1990.

\bibitem{hochster_spectral}
Melvin Hochster.
\newblock Prime ideal structure in commutative rings.
\newblock {\em Trans. Amer. Math. Soc.}, 142:43--60, 1969.

\bibitem{hub-kneb}
Roland Huber and Manfred Knebusch.
\newblock On valuation spectra.
\newblock In {\em Recent advances in real algebraic geometry and quadratic
  forms ({B}erkeley, {CA}, 1990/1991; {S}an {F}rancisco, {CA}, 1991)}, volume
  155 of {\em Contemp. Math.}, pages 167--206. Amer. Math. Soc., Providence,
  RI, 1994.

\bibitem{kaplansky}
Irving Kaplansky.
\newblock {\em Commutative {R}ings}.
\newblock The University of Chicago Press, Chicago, Ill.-London, revised
  edition, 1974.

\bibitem{krull-estval}
Wolfgang Krull.
\newblock \"{U}ber einen {E}xistenzsatz der {B}ewertungstheorie.
\newblock {\em Abh. Math. Sem. Univ. Hamburg}, 23:29--35, 1959.

\bibitem{nagata_localrings}
Masayoshi Nagata.
\newblock {\em Local {R}ings}.
\newblock Interscience Tracts in Pure and Applied Mathematics, No. 13.
  Interscience Publishers a division of John Wiley \& Sons\, New York-London,
  1962.

\bibitem{olberding_noetherianspaces}
Bruce Olberding.
\newblock Noetherian spaces of integrally closed rings with an application to
  intersections of valuation rings.
\newblock {\em Comm. Algebra}, 38(9):3318--3332, 2010.

\bibitem{olberding_affineschemes}
Bruce Olberding.
\newblock Affine schemes and topological closures in the {Z}ariski-{R}iemann
  space of valuation rings.
\newblock {\em J. Pure Appl. Algebra}, 219(5):1720--1741, 2015.

\bibitem{olberding_topasp}
Bruce Olberding.
\newblock Topological aspects of irredundant intersections of ideals and
  valuation rings.
\newblock In {\em Multiplicative Ideal Theory and Factorization Theory:
  Commutative and Non-Commutative Perspectives}. Springer Verlag, 2016.

\bibitem{PerTransc}
Giulio Peruginelli.
\newblock Transcendental extensions of a valuation domain of rank one.
\newblock {\em Proc. Amer. Math. Soc.}, 145(10):4211--4226, 2017.

\bibitem{fundstatz}
Giulio Peruginelli and Dario Spirito.
\newblock Extending valuations to the field of rational functions using
  pseudo-monotone sequences.
\newblock {\em J. Algebra}, to appear.

\bibitem{schwartz-compactification}
Niels Schwartz.
\newblock Compactification of varieties.
\newblock {\em Ark. Mat.}, 28(2):333--370, 1990.

\bibitem{seiden-dim}
Abraham Seidenberg.
\newblock A note on the dimension theory of rings.
\newblock {\em Pacific J. Math.}, 3:505--512, 1953.

\bibitem{ZarNoeth}
Dario Spirito.
\newblock Non-compact subsets of the {Z}ariski space of an integral domain.
\newblock {\em Illinois J. Math.}, 60(3-4):791--809, 2016.

\bibitem{localizzazioni}
Dario Spirito.
\newblock Topological properties of localizations, flat overrings and
  sublocalizations.
\newblock {\em J. Pure Appl. Algebra}, 223(3):1322--1336, 2019.

\bibitem{ZarNoeth2}
Dario Spirito.
\newblock When the {Z}ariski space is a {N}oetherian space.
\newblock {\em Illinois J. Math.}, 63(2):299--316, 2019.

\bibitem{willard_gentop}
Stephen Willard.
\newblock {\em General {T}opology}.
\newblock Dover Publications, Inc., Mineola, NY, 2004.
\newblock Reprint of the 1970 original [Addison-Wesley, Reading, MA].

\bibitem{zariski_sing}
Oscar Zariski.
\newblock The reduction of the singularities of an algebraic surface.
\newblock {\em Ann. of Math. (2)}, 40:639--689, 1939.

\bibitem{zariski_comp}
Oscar Zariski.
\newblock The compactness of the {R}iemann manifold of an abstract field of
  algebraic functions.
\newblock {\em Bull. Amer. Math. Soc.}, 50:683--691, 1944.

\bibitem{zariski_samuel_II}
Oscar Zariski and Pierre Samuel.
\newblock {\em Commutative {A}lgebra. {V}ol. {II}}.
\newblock Springer-Verlag, New York, 1975.
\newblock Reprint of the 1960 edition, Graduate Texts in Mathematics, Vol. 29.

\end{thebibliography}
\end{document}